\documentclass[12pt,a4paper]{article}
\usepackage[utf8x]{inputenc}
\usepackage{amsmath}
\usepackage{amsfonts}
\usepackage{amssymb}
\usepackage{amsthm}
\usepackage{mathrsfs}
\usepackage{fancyhdr}
\usepackage{graphicx}
\usepackage[usenames,x11names]{xcolor}
\usepackage{arabtex}
\usepackage{framed}
\usepackage[top=2cm,bottom=2cm,margin=2cm]{geometry}
\usepackage{cancel}
\usepackage{array}   %%%% Pour les fonctionalités assez évoluées d'un tableau

\usepackage[dvipdfm,pdfstartview=FitH,colorlinks=true,linkcolor=Red4,urlcolor=Red4,%
filecolor=green,citecolor=red]{hyperref}

\title{Arithmetic properties of the Genocchi numbers and their generalization}
\author{\sc Bakir FARHI \\
Laboratoire de Mathématiques appliquées \\
Faculté des Sciences Exactes \\
Université de Bejaia, 06000 Bejaia, Algeria \\[1mm]
\href{mailto:bakir.farhi@gmail.com}{bakir.farhi@gmail.com} \\[1mm]
\url{http://farhi.bakir.free.fr/}
}
\date{}

\let\up=\textsuperscript

\def\Q{{\mathbb Q}}

\def\N{{\mathbb N}}
\def\Z{{\mathbb Z}}

\def\int{\mathrm{Int}}

\def\den{\mathrm{den}}    %%%% Dénominateur
\def\num{\mathrm{num}}    %%%% Numérateur
\def\restmod#1#2{#1\ (\mathrm{mod}\ #2)} %%% Congruences

\def\EMdash{\leavevmode\hbox to 10.6mm{\vrule height .63ex depth -.59ex
    width 10mm\hfill}}

\theoremstyle{plain}
\numberwithin{equation}{section}
\newtheorem{thm}{Theorem}[section]

\newtheorem{rmq}[thm]{Remark}
\newtheorem{prop}[thm]{Proposition}
\newtheorem{coll}[thm]{Corollary}

\begin{document}
\maketitle
\begin{abstract}
This note is devoted to establish some new arithmetic properties of the generalized Genocchi numbers $G_{n , a}$ ($n \in \mathbb{N}$, $a \geq 2$). The resulting properties for the usual Genocchi numbers $G_n = G_{n , 2}$ are then derived. We show for example that for any even positive integer $n$, the Genocchi number $G_n$ is a multiple of the odd part of $n$.   
\end{abstract}

\section{Introduction and Notations}
Throughout this note, we let $\N^*$ denote the set of positive integers. For a given prime number $p$, we let $\vartheta_p$ denote the usual $p$-adic valuation. For a given rational number $r$, we let $\num(r)$ and $\den(r)$ respectively denote the numerator and the denominator of $r$; precisely, if $d$ is the smallest positive integer satisfying $d r \in \Z$ then we have: $\den(r) := d$ and $\num(r) := d r$. Next, for given positive integers $a$ and $n$, we let $\pi_a(n)$ denote the greatest positive divisor of $n$ which is coprime with $a$. Then, it is immediate that:
$$
\pi_a(n) ~=~ \prod_{\begin{subarray}{c}
p \text{ prime} \\
p \nmid a
\end{subarray}} p^{\vartheta_p(n)} .
$$
In the particular case $a = 2$, $\pi_a(n)$ is called \emph{the odd part} of $n$; it is the greatest odd positive divisor of $n$.

Additionally, we need to extend the congruences in $\Z$ to $\Q$ as follows:
\begin{quotation}
\noindent For $a , b \in \Q$ and $n \in \N^*$, we write $a \equiv \restmod{b}{n}$ if the numerator of the rational number $(a - b)$ is a multiple of $n$.
\end{quotation}
In the same context, an equivalent meaning of the congruence $a \equiv \restmod{b}{n}$ consists to say that for any prime divisor $p$ of $n$, we have $\vartheta_p(a - b) \geq \vartheta_p(n)$. We can easily check that some properties (but not all) of the congruences in $\Z$ become valid in $\Q$. For example, we can sum side to side several congruences in $\Q$ with a same modulus. Also, if $a , b , c \in \Q$ and $n \in \N^*$ such that $\den(c)$ is coprime with $n$ then we have:
$$
a \equiv \restmod{b}{n} ~\Longrightarrow~ a c \equiv \restmod{b c}{n} .
$$
However, we cannot multiply (side to side) several congruences in $\Q$ with a same modulus (contrary to the congruences in $\Z$).

Further, we denote by $\Q[[t]]$ the ring of formal power series in $t$ with coefficients in $\Q$. If an element $f$ of $\Q[[t]]$ is represented as
$$
f(t) ~=~ \sum_{n = 0}^{+ \infty} a_n \frac{t^n}{n!}
$$
(where $a_n \in \Q$, $\forall n \in \N$), we call the $a_n$'s \emph{the differential coefficients} of $f$ (because each $a_n$ is the $n$\up{th} derivative of $f$ at $0$). If the $a_n$'s are all integers, we say that $f$ is an \emph{IDC-series} (IDC abbreviates the expression ``with Integral Differential Coefficients''). Many usual functions are IDC-series; we can cite for example the functions $x \mapsto e^x$, $x \mapsto \sin{x}$, $x \mapsto \cos{x}$, $x \mapsto \ln(1 + x)$, and so on. We can easily see that the sum and the product of two IDC-series become an IDC-series. Other important properties of the IDC-series are given by the author \cite{far} who used them to establish his generalization of the Genocchi theorem (see below). One of those properties is given by the following proposition:
\begin{prop}[{\cite[Corollary 2.2]{far}}]\label{p1}
Let $f(t) = \sum_{n = 0}^{+ \infty} a_n \frac{t^n}{n!}$ be an IDC-series with $a_0 \neq 0$. Then the formal power series $\frac{a_0}{f(a_0 t)}$ is also an IDC-series. 
\end{prop} 
Showing that a given function is an IDC-series is, in general, not easy. For example, the statement that the function $t \mapsto \frac{2 t}{e^t + 1}$ is an IDC-series constitute a profound theorem due to Genocchi \cite{gen}. The differential coefficients of the last function (which thus are all integers) are called \emph{the Genocchi numbers} and denoted by $G_n$ ($n \in \N$); so we have:
$$
\frac{2 t}{e^t + 1} ~=~ \sum_{n = 0}^{+ \infty} G_n \frac{t^n}{n!} .
$$
For a study of the Genocchi numbers, the reader can consult \cite{com,dum,far,gan,gen,rio}.

Very recently, the author \cite{far} has generalized the Genocchi numbers by considering, for a given integer $a \geq 2$, the function $t \mapsto \frac{a t}{e^{(a - 1) t} + e^{(a - 2) t} + \dots + e^t + 1}$ and its expansion into a power series:
$$
\frac{a t}{e^{(a - 1) t} + e^{(a - 2) t} + \dots + e^t + 1} ~=~ \sum_{n = 0}^{+ \infty} G_{n , a} \frac{t^n}{n!} .
$$
So, for $a = 2$, we simply obtain the usual Genocchi numbers; that is $G_{n , 2} = G_n$ ($\forall n \in \N$). The main result of \cite{far} then states that the $G_{n , a}$'s are all integers, which generalizes the Genocchi theorem. The Genocchi numbers as well as their generalization are ultimately related to the Bernoulli numbers ${(B_n)}_{n \in \N}$, which can be defined by the exponential generating function:
$$
\frac{t}{e^t - 1} ~=~ \sum_{n = 0}^{+ \infty} B_n \frac{t^n}{n!}
$$
(see e.g., \cite{com,nie}). Among others, we have the following connection between the generalized Genocchi numbers and the Bernoulli numbers:
\begin{prop}[{\cite[Proposition 4.1]{far}}]\label{p2}
For any positive integers $a$ and $n$, with $a \geq 2$, we have:
$$
G_{n , a} ~=~ \sum_{k = 0}^{n - 1} \binom{n}{k} B_k a^k .
$$
\end{prop}
Furthermore, the Bernoulli numbers have many arithmetic properties. The most famous is certainly that given by the von Staudt-Clausen theorem (see e.g., \cite{har,nie}), recalled below:
\begin{thm}[von Staudt-Clausen]
For any even positive integer $n$, we have that:
$$
B_n + \sum_{\begin{subarray}{c}
p \text{ prime} \\
(p - 1) \mid n
\end{subarray}} \frac{1}{p} ~\in~ \Z .
$$
\end{thm} 
\noindent Because it is known that $B_0 = 1$, $B_1 = - \frac{1}{2}$ and $B_n = 0$ for any odd integer $n \geq 3$, we can derive from the von Staudt-Clausen theorem the following immediate corollary which is useful for our purpose:
\begin{coll}\label{c1}
For any natural number $n$ and any prime number $p$, we have that:
$$
\vartheta_p(B_n) \geq - 1 .
$$
\end{coll}
The purpose of this note consists to establish arithmetic properties of the numbers $G_{n , a}$ (in addition to their integrality, previously proven in \cite{far}). Precisely, we obtain for the $G_{n , a}$'s a nontrivial simple divisor and a nontrivial congruence modulo $a$.

\section{The results and the proofs}
We begin by establishing nontrivial divisors for the generalized Genocchi numbers. 

\begin{thm}\label{t1}
Let $a \geq 2$ be an integer. Then for all positive integer $n$, the integer $G_{n , a}$ is a multiple of the integer $\pi_a(n)$.
\end{thm}

\begin{proof}
By applying Proposition \ref{p1} to the IDC-series $t \mapsto e^{(a - 1) t} + e^{(a - 2) t} + \dots + e^t + 1$, we find that the function $t \stackrel{f}{\mapsto} \frac{a}{e^{(a - 1) a t} + e^{(a - 2) a t} + \dots + e^{a t} + 1}$ is an IDC-series. Now, let us explicitly express the differential coefficients of $f$ in terms of the generalized Genocchi numbers. We have:
\begin{eqnarray*}
\frac{a}{e^{(a - 1) t} + e^{(a - 2) t} + \dots + e^t + 1} & = & \frac{1}{t} \cdot  \frac{a t}{e^{(a - 1) t} + e^{(a - 2) t} + \dots + e^t + 1} \\[2mm]
& = & \frac{1}{t} \sum_{n = 1}^{+ \infty} G_{n , a} \frac{t^n}{n!} \\[2mm]
& = & \sum_{n = 1}^{+ \infty} \frac{G_{n , a}}{n} \cdot \frac{t^{n - 1}}{(n - 1)!} .
\end{eqnarray*}
Then, by substituting $t$ by $a t$, we get
$$
f(t) ~=~ \sum_{n = 1}^{+ \infty} \frac{a^{n - 1} G_{n , a}}{n} \cdot \frac{t^{n - 1}}{(n - 1)!} .
$$
So, since $f$ is an IDC-series, we derive that $\frac{a^{n - 1} G_{n , a}}{n} \in \Z$ ($\forall n \in \N^*$); that is
$$
n \mid a^{n - 1} G_{n , a} ~~~~~~~~~~ (\forall n \in \N^*) .
$$
Next, since $\pi_a(n)$ is a divisor of $n$ which is coprime with $a$, we derive that $\pi_a(n) \mid a^{n - 1} G_{n , a}$ and $\pi_a(n)$ is coprime with $a^{n - 1}$. Then, by Gauss's lemma, we conclude that $\pi_a(n)$ divides $G_{n , a}$, as required.
\end{proof}

By taking $a = 2$ in Theorem \ref{t1}, we immediately deduce the following corollary which concerns the usual Genocchi numbers:

\begin{coll}
For all positive integer $n$, the Genocchi number $G_n$ is a multiple of the odd part of $n$. \hfill $\square$
\end{coll}

We now turn to study the remainders of the generalized Genocchi numbers $G_{n , a}$ modulo $a$. The main result in this direction is the following:

\begin{thm}\label{t2}
For any integers $a$ and $n$, both greater than $1$, we have
$$
G_{n , a} ~\equiv~ \restmod{1 - \frac{n}{2} \, a}{a} .
$$
\end{thm}

\begin{proof}
Let $a$ and $n$ be two integers, both greater than $1$. According to Proposition \ref{p2}, we have that:
\begin{eqnarray}
G_{n , a} & = & \sum_{k = 0}^{n - 1} \binom{n}{k} B_k a^k \notag \\
& = & 1 - \frac{n}{2} a + \sum_{k = 2}^{n - 1} \binom{n}{k} B_k a^k . \label{eq1}
\end{eqnarray} 
Next, by using Corollary \ref{c1}, we have for any $k \in \{2 , 3 , \dots , n - 1\}$ and any prime divisor $p$ of $a$ (so $\vartheta_p(a) \geq 1$):
\begin{eqnarray*}
\vartheta_p\left(B_k a^k\right) & = & \vartheta_p\left(B_k\right) + (k - 1) \vartheta_p(a) + \vartheta_p(a) \\
& \geq & - 1 + (k - 1) + \vartheta_p(a) \\
& = & k - 2 + \vartheta_p(a) \\
& \geq & \vartheta_p(a) ,
\end{eqnarray*}
showing that $B_k a^k \equiv \restmod{0}{a}$ ($\forall k \in \{2 , 3 , \dots , n - 1\}$). By substituting these congruences into \eqref{eq1}, we conclude to the required congruence $G_{n , a} \equiv \restmod{1 - \frac{n}{2} a}{a}$.
\end{proof}

From Theorem \ref{t2}, we immediately derive the following corollary:
\begin{coll}\label{c2}
Let $a \geq 2$ be an integer.
\begin{itemize}
\item If $a$ is odd then we have for any positive integer $n$:
$$
G_{n , a} ~\equiv~ \restmod{1}{a} .
$$
\item If $a$ is even then we have for any integer $n \geq 2$:
\begin{equation}
G_{n , a} ~\equiv~ \restmod{\begin{cases}
1 & \text{if $n$ is even} \\
1 + \frac{a}{2} & \text{if $n$ is odd}
\end{cases}}{a} . \tag*{$\square$}
\end{equation} 
\end{itemize}
\end{coll}

\begin{rmq}
By applying Corollary \ref{c2} for $a = 2$, we obtain the well-known result stating that the usual Genocchi numbers of even positive orders are all odd.
\end{rmq}

We conclude this note by the following corollary which is an immediate consequence of Corollary \ref{c2} above.

\begin{coll}
For any integers $a$ and $n$, both greater than $1$, we have:
$$
\gcd\left(G_{n , a} , a\right) ~\in~ \{1 , 2\} .
$$
Besides, the equality $\gcd(G_{n , a} , a) = 2$ holds if and only if $a \equiv \restmod{2}{4}$ and $n$ is odd. \hfill $\square$
\end{coll}

\end{document}